\documentclass[11pt]{article}
\usepackage{anysize}\marginsize{3.5cm}{3.5cm}{1.3cm}{2cm}
\makeatletter\@ifundefined{pdfpagewidth}{}{\pdfpagewidth=21.0cm\pdfpageheight=29.7cm}\makeatother 
\usepackage{amssymb,amsmath}

\makeatletter
\let\orig@item=\@item \def\@item[#1]{\orig@item[\rm #1]}
\renewenvironment{abstract}{\begin{quote}\footnotesize\textbf{\abstractname.}}{\end{quote}\bigskip}
\renewcommand\@seccntformat[1]{\csname the#1\endcsname.\enspace}
\renewcommand\section{\@startsection{section}{1}{\z@}{-2\baselineskip plus 0.5\baselineskip minus 0.5\baselineskip}{0.75\baselineskip plus 0.5\baselineskip}{\normalsize\bfseries\centering}}
\renewcommand\subsection{\@startsection{subsection}{2}{\z@}{-2\baselineskip plus 0.5\baselineskip minus 0.5\baselineskip}{0.75\baselineskip plus 0.5\baselineskip}{\normalsize\bfseries\centering}}
\renewcommand\paragraph{\@startsection{paragraph}{4}{\z@}{1\baselineskip}{-0.5em}{\normalsize\bfseries}}
\let\origcaption=\caption \renewcommand\caption[1]{\parbox{0.66\textwidth}{\origcaption{#1}}}

\renewcommand\@begintheorem[2]{\trivlist\item[\hskip\labelsep{\bfseries#1 #2.}]\it}
\renewcommand\@opargbegintheorem[3]{\trivlist\item[\hskip\labelsep{\bfseries#1 #2}] {\bfseries(#3).}\enspace\it\ignorespaces}
\makeatother

\newtheorem{satz}{Satz}[section]
\makeatletter\@addtoreset{equation}{satz}\makeatother

\newtheorem{theorem}[satz]{Theorem}

\newtheorem{cor}[satz]{Corollary}
\newtheorem{example}[satz]{Example}

\newtheorem{lemma}[satz]{Lemma}
\newtheorem{proposition}[satz]{Proposition}
\newtheorem{corollary}[satz]{Corollary}

\newtheorem{remark}[satz]{Remark}

\newtheorem{introtheorem}{Theorem}
\newenvironment{proof}[1][Proof]{\trivlist\item[\hskip\labelsep{\it #1.}]}{\hspace*{\fill}$\Box$\endtrivlist}

\newcommand\subjclass[1]{{\renewcommand\thefootnote{}\footnotetext{2000 \textit{Mathematics Subject Classification:} #1.}}}
\newcommand\keywords[1]{{\renewcommand\thefootnote{}\footnotetext{\textit{Keywords.} #1.}}}

\newcommand\engq[1]{`#1'}
\newcommand\engqq[1]{``#1''}

\renewcommand\emptyset{\varnothing}  
\renewcommand\ge{\geqslant}  
\renewcommand\le{\leqslant}  
\renewcommand\epsilon{\varepsilon}
\renewcommand\phi{\varphi}
\renewcommand\bar{\overline}

\renewcommand\tilde{\widetilde}

\renewcommand\P{\mathbb P}

\newcommand\be{\begingroup\arraycolsep=0.13888em\begin{eqnarray*}}
\newcommand\ee{\end{eqnarray*}\endgroup}

\newcommand\set[1]{\left\{#1\right\}}

\newcommand\with{\ \vrule\ }

\newcommand\maxmatrcols{10}
\newlength\matrcolsep \matrcolsep=\arraycolsep
\newcommand\matr[1]{{\arraycolsep=\matrcolsep\left(\begin{array}{*{\maxmatrcols}{c}}#1\end{array}\right)}}
\newcommand\tline{\noalign{\vskip0.4ex}\hline\noalign{\vskip0.65ex}}

\newcommand\Q{\mathbb Q}
\newcommand\R{\mathbb R}
\newcommand\Z{\mathbb Z}

\newcommand\newop[2]{\newcommand#1{\mathop{\rm #2}\nolimits}}
\newop\NS{NS} 
\newop\Amp{Amp} 
\newop\Bl{Bl} 
\newop\End{End} 
\newop\Nef{Nef}
\newop\vol{vol}
\newop\Bigcone{Big} 
\newop\interior{int}

\newcommand\PtP{\P^1\times\P^1}

\newop{\rang}{rang}
\newop{\dv}{div}
\newop{\id}{id}
\newop{\ord}{ord}
\newop{\Div}{Div}
\newop{\Face}{Face}
\newop{\Vol}{Vol}
\newop{\Neg}{Neg}
\newop{\Cox}{Cox}
\newop{\Null}{Null}
\newop{\NE}{\overline{{NE}}}

\begin{document}

   \title{Volumes of Zariski chambers}
   \author{Thomas Bauer and David Schmitz}
   \date{May 15, 2012}
   \maketitle
   \thispagestyle{empty}


\begin{abstract}
   Zariski chambers are natural pieces into which the big cone of
   an algebraic surface decomposes. They have so far been studied
   both from a geometric and from a combinatorial perspective.
   In the present paper we
   complement the picture with a \emph{metric}
   point of view
   by studying a suitable notion of \emph{chamber sizes}.
   Our first result gives a precise condition for the nef cone
   volume to be finite and provides a method for computing it
   inductively. Our second result determines the volumes of
   arbitrary Zariski chambers from nef cone volumes of blow-downs.
   We illustrate the
   applicability of this method by explicitly
   determining the chamber volumes
   on Del Pezzo and other
   anti-canonical surfaces.
\end{abstract}


\section*{Introduction}

   In this note we study the natural decomposition of the big cone
   on a smooth projective surface into Zariski chambers
   as introduced in \cite{BKS}. Being convex cones (and therefore
   non-compact) the chambers
   cannot a priori be compared in terms of size. The purpose of this
   note is to introduce a notion of volume of Zariski chambers,
   find criteria for finiteness of chamber volumes, and to
   show how chamber volumes can be calculated explicitly.

   Let $X$ be a smooth projective surface. We consider the
   convex cone $\Bigcone(X)$ in the N\'eron-Severi vector space
   $N_\R^1(X) := N^1(X) \otimes \R$ spanned by the classes of
   big divisors on $X$. By the main result of \cite{BKS}, it
   admits a locally finite decomposition
   into locally polyhedral subcones with the following properties:
   \begin{itemize}
   \item
      the support of the negative part in the Zariski decomposition
      is constant on each subcone,
   \item
      the volume function is given by a quadratic polynomial on each
      subcone, and
   \item
      the stable base loci are constant in the interior of each subcone.
   \end{itemize}
   On account of the first listed property the subcones are called
   \emph{Zariski chambers}.
   For a big and nef divisor $P$ on $X$ we consider the set $\Null(P)$ of
   irreducible curves having intersection zero with $P$. The \emph{chamber} $\Sigma_P$
   corresponding to $P$ consists of all big divisors whose negative part in the
   Zariski decomposition has support $\Null(P)$. For example, if $P$ is
   ample, then $\Null(P)$ is empty, thus $\Sigma_P$ is the intersection
   of the big cone with the nef cone, the \emph{nef chamber}.

   Zariski chambers have first been studied with respect to
   \emph{geometric} aspects in \cite{BKS}. A \emph{combinatorial} point of view
   has been taken in \cite{BFN}, where a method for determining
   the number of chambers was presented. In the present paper we
   would like to complement the picture with a \emph{metric}
  point of view: we ask whether one can measure the \engq{size}
  of a chamber -- with the aim of introducing a quantity that tells, intuitively,
  \engq{how far} a line bundle can be moved without changing its stable
  base locus. A natural starting point is an
   invariant that was introduced in \cite{Pe} to measure the nef
   cones of Del Pezzo surfaces. It has recently been studied by
   Derenthal in a series of papers
   (see \cite{D}, \cite{DJT}, and also
   \cite{D3}, \cite{D4}, and \cite{D5}).
   We extend
   this notion in order to measure arbitrary Zariski chambers on
   arbitrary surfaces.

   Note to begin with that the N\'eron-Severi vector
   space~$N_\R^1(X)$ can be equipped with
   a canonical Lebesgue
   measure $ds$ (not depending on the choice of a basis or on an
   isomorphism with $\R^n$) by requiring that the lattice
   $N^1(X)$ has covolume~1, i.e., by normalizing it in such a way
   such that the fundamental parallelotope of $N^1(X)$ with
   respect to a fixed basis has $ds$-volume 1. The transformation
   formula together with the fact that a matrix transforming
   lattice bases into lattice bases has determinant~$\pm 1$
   guarantees the independence of the choice of a lattice basis.

   Consider then for a convex cone $C\subset N_\R^1(X)$ in the
   N\'eron-Severi vector space of a smooth projective surface $X$ the
   set
   $$
      \mathcal C_C := \bar C \cap (-K_X)^{\le 1},
   $$
   where $(-K_X)^{\le 1}$ denotes the half space of divisors having intersection
   at most 1 with the anticanonical divisor $-K_X$ on $X$.
   The \textit{cone volume} $\Vol(C)$ is defined to be the $ds$-volume of the
   set $\mathcal C_C$.
   Note that the set $\mathcal C_C$ need not be compact, hence infinite cone
   volume can occur. Our first main result gives a necessary and sufficient condition
   for the nef cone to have finite volume, and moreover states that in this case
   the volume can be computed inductively:

   \begin{introtheorem}
      Let $X$ be a smooth projective surface with Picard number $\rho$.
      The nef cone volume $\Vol(\Nef(X))$ is finite if and
      only if the anticanonical divisor $-K_X$ on $X$ is big. In this case there exists a divisor
      $D\in N^1_\R(X)$ such that the nef cone volume is given by
      $$
         \Vol(\Nef(X)) = \frac1{\rho}\cdot \sum_{E}{(D\cdot E)\cdot\Vol(\Nef(\pi_E(X)))},
      $$
      with the sum taken over all $(-1)$-curves $E$ in $X$ and $\pi_E$
      denoting the contraction of $E$.
   \end{introtheorem}
   \noindent
   When applied to del Pezzo surfaces, one obtains the explicit values computed by
   Derenthal in \cite{D} (see Example \ref{ex: dP} below). Other examples will be the calculation
   of nef cone volumes on surfaces obtained by blowing up the
   projective plane
   in points on a line and in infinitely near points
   (see Section \ref{sect: Bl up}).

   Note that the nef cone is always the closure of the nef chamber (non-big nef divisors
   have self-intersection zero, hence lie on the boundary of the nef cone). Therefore,
   the volume of the nef cone equals the volume of the Zariski chamber $\Sigma_H$ for
   any ample divisor $H$.
   The second main result of this note  deals with the volumes of the
   remaining Zariski chambers.
   We show that in fact knowledge of nef chamber volumes on the surfaces resulting
   from the contraction $\pi_S:X\to Y$ of sets $S$ of pairwise disjoint $(-1)$-curves
   suffices to calculate the volumes of arbitrary chambers:
   \begin{introtheorem}
      Let $X$ be a smooth projective surface with Picard number $\rho$, and let $P$ be
      a big and nef divisor on $X$ and $S=\set{E_1,\ldots,E_s}=\Null(P)$.
      Either $S$ contains a curve
      of self-intersection less than $-1$ and $\Vol(\Sigma_P)=\infty$, or $S$ consists
      of $s$ pairwise disjoint $(-1)$-curves and
      $$
         \Vol(\Sigma_P) = \frac{(\rho- s)!}{\rho!}\Vol(\Nef(\pi_S(X))).
      $$
   \end{introtheorem}
   We give two applications of Theorem 2: first, we use it in
   Section \ref{delPezzo} to determine all chamber volumes on del Pezzo
   surfaces. An interesting aspect here is that chambers of the
   same support size (the number $s$ appearing in Theorem 2) can
   lead to non-isomorphic surfaces by
   blow-down -- and precisely this geometric difference can be
   detected from the chamber volumes. As a second application, we
   study chambers on certain surfaces with big but non-ample
   anticanonical divisor (Sections \ref{s:big} and \ref{subsect:infinitely-near}).

   Throughout this paper we work over the complex numbers.
   We would like to thank the referee for his
   valuable comments.


\section{Zariski chamber volumes}

   For the Zariski chamber decomposition we follow the notation from \cite{BKS}:
   for a big and nef divisor $P$ on $X$ we consider the set $\Null(P)$ of
   irreducible curves having intersection zero with $P$. Note that by the
   index theorem the intersection matrix of the curves in $\Null(P)$ must be
   negative definite. The chamber  $\Sigma_P$ corresponding to $P$ is defined as the set
   of all big divisor classes $D$ such that the support  of the negative part in the
   Zariski decomposition of $D$, denoted by $\Neg(D)$, equals $\Null(P)$.
   In \cite{BKS} it is shown that two chambers $\Sigma_P$ and $\Sigma_{P'}$
   either coincide or are disjoint, and that all of the big cone is covered by
   the union of all Zariski chambers. Furthermore we consider the set $\Face(P)$
   defined as the intersection of the nef cone with $\Null(P)^\bot$, the set of
   divisor classes $D$ having intersection zero with all elements of $\Null(P)$.
   If $\Face(P)$ is contained in $\Bigcone(X)$, then $\Face(P)$
   turns out to be the lowest dimensional face of the nef cone containing $P$
   (see \cite[Remark 1.5]{BKS}). We will frequently use the following
\begin{proposition}\label{prop: chamber structure}
   The closure $\overline{\Sigma}_P$ of the Zariski chamber corresponding to a big and nef
   divisor $P$ is the convex cone spanned by $\Face(P)$ and the curves in $\Null(P)$.
\end{proposition}

\begin{proof}
   This follows from \cite[Proposition 1.10]{BKS} by taking the closure.
\end{proof}

   Upon choosing a lattice basis of $N^1(X)$, the euclidean vector space $N^1_\R(X)$
   is equipped with a norm $\|\cdot\|_2$ coming from the scalar product.
   For the contraction $\pi_S:X\to Y$ of a set $S$ of disjoint $(-1)$-curves we consider
   the map $\pi_S^\ast:N^1_\R(Y)\to N^1_\R(X)$ given by pulling back divisors. The
   pull-back of any fixed lattice basis $B$ of $N^1(Y)$ is extended to a lattice basis $B'$
   of $N^1(X)$ by the elements of $S$. The map $\pi_S^\ast$ embeds $N^1_\R(Y)$  into
   $N^1_R(X)$ isometrically with regard to the bases $B$ and $B'$. Note furthermore that
   for any $E\in S$ the hyperplane $E^\bot$ given by divisor classes having intersection zero
   with $E$ coincides with the hyperplane of vectors orthogonal to $E$ with respect to the scalar
   product once $B'$ has been fixed as lattice basis.

\begin{proposition}\label{intersecting is contracting}
   Let $S=\set{E_1,\ldots,E_s}$ be a set of pairwise disjoint $(-1)$-curves on a smooth
   projective
   surface $X$. Then
   $$
      \Nef(X) \cap S^\bot = \pi_S^\ast(\Nef(\pi_S(X))).
   $$
\end{proposition}

\begin{proof}
   We prove the result for the case $s=1$ and the assertion follows inductively.
   Consider the surjective morphism of smooth surfaces
   $$
      \pi_E: X \to Y
   $$
   given by the contraction of the $(-1)$-curve $E\in S$.
   Any divisor
   $D\in E^\bot$ on $X$ is the pull-back of a divisor $\bar D$ on $\pi_E(X)$, and for all
   divisors $F$ on $X$ we have the projection formula
   $$
      D\cdot  F = \bar D\cdot \pi_E (F),
   $$
   implying that $D$ is nef if and only if $\bar D$ is. Furthermore, the pull-back of any
   curve in $Y$ obviously lies in the hyperplane $E^\bot$.
\end{proof}

   \begin{cor}

      Let $P$ be a big and nef divisor such that all curves in $\Null(P)$ are $(-1)$-curves
      and let $S=\set{E_1,\ldots,E_s}$ be a subset of $\Null(P)$. Then
      $$
         \Sigma_P \cap S^\bot = \pi_S^\ast(\Sigma_{\pi_S(P)}).
      $$
   \end{cor}

\begin{proof}
   As above it suffices to consider the case $s=1$.

   Remember that $\Face(P)$ is given as
   the intersection of $\Null(P)^\bot$ and the nef cone $\Nef(X)$. Now, the intersection of
   the nef cone on $X$ with the hyperplane $E^\bot$ corresponds to the nef cone on
   $\pi_E(X)$  via $\pi_E^\ast$ by virtue of the proposition above. On the other hand
   $$
      \Null(P)\cap E^\bot = \Null(P) - \set{E} = \pi_E^\ast(\Null(\pi_E(P))),
   $$
   which implies the identity $\Face(P)=\pi^\ast(\Face(\pi_E(P)))$. This, together with
   Proposition \ref{prop: chamber structure}, completes the proof.
\end{proof}

   Let us now prove our first main result, which shows that the calculation of volumes of Zariski chambers
   can be reduced to the calculation of nef cone volumes.

\begin{theorem}\label{th chamber volume}
   Let $X$ be a smooth projective surface with Picard number $\rho$, and let $P$ be a big and nef
   divisor on $X$ and $S=\set{E_1,\ldots,E_s}=\Null(P)$. Either $S$ contains a curve
   of self-intersection less than $-1$ and $\Vol(\Sigma_P)=\infty$, or $S$ consists
   of $s$ pairwise disjoint $(-1)$-curves and
   $$
      \Vol(\Sigma_P) = \frac{(\rho- s)!}{\rho!}\Vol(\Nef(\pi_S(X))).
   $$
\end{theorem}

\begin{proof}
   The case $s=0$ is trivial, so assume that $S$ is non-empty.
   Note that since $S=\Null(P)$ has negative definite intersection matrix, the alternatives
   really constitute a dichotomy. Now, suppose there exists an irreducible curve
   $C\in S$ with $C^2 < -1$. By adjunction we have
   $$
      -K_X\cdot C \le 0,
   $$
   i.e., the hyperplane $(-K_X)^{=1}$ does not intersect the ray $\R^+\cdot [C]$ which is
   contained in $\Sigma_P$. Therefore, $\Sigma_P$ has infinite volume.

   If $S$ consists of pairwise disjoint $(-1)$-curves, we know by Proposition
   \ref{prop: chamber structure} that $\bar{\Sigma}_P$ is the convex cone spanned by $\Face(P)$ and the
   curves of $S$. By Proposition \ref{intersecting is contracting}
   we have
   \begin{eqnarray*}
      \Face(P) &=& \Nef(X) \cap S^\bot\\
      &=& \pi_S^\ast(\Nef(\pi_S(X))).
   \end{eqnarray*}
   Additionally, for $E\in S$ and for a divisor $D$ on $X$ with $D\cdot E=0$ we have
   $-K_X\cdot D= -K_{\pi_s(X)}\cdot \bar D$, where
   $D= \pi_S^\ast(\bar D)$, implying that $\mathcal C_{\Nef(X)}\cap S^\bot$
   can be identified via $\pi_S^\ast$ with $\mathcal C_{\Nef(\pi_S(X))}$.
   For a lattice basis $C_1,\ldots,C_{\rho-s}$ of $N^1(\pi_S(X))$
   the vectors
   $$
      \pi_S^\ast(C_1),\ldots,\pi_S^\ast(C_{\rho-s}),E_1,\ldots,E_s
   $$
   form a basis of the lattice $N^1(X)$.
   Consider the polytopes
   \begin{eqnarray*}
      P_1 &:=& \mbox{conv} (\mathcal C_{\Face(P)},E_1)\\
      P_j &:=& \mbox{conv} (P_{j-1}, E_j) \qquad 2 \le j \le s,
   \end{eqnarray*}
   where $\mbox{conv}$ denotes the convex hull.
   Each $P_j$ is a $(\rho-s+j)$-dimensional pyramid with base $P_{j-1}$ and vertex
   $E_j$. The vector $E_j$ is perpendicular to the subspace
   $E_1^\bot\cap \ldots\cap E_{j-1}^\bot$ containing $P_{j-1}$. Furthermore, due to
   the choice of the basis we have $\|E_j\|_2=1$. Therefore the pyramid $P_j$ has volume
   $$
      \Vol(P_j) = \frac1{\rho-s+j} \Vol(P_{j-1}).
   $$
   Iterating this calculation eventually yields
   \begin{eqnarray*}
      \Vol(\Sigma_P) = \Vol(P_s)
      &=&    \frac1{\rho-s+1}\cdot\ldots\cdot\frac1{\rho-s+s}\Vol(\Face(P))\\
      &=& \frac{(\rho - s)!}{\rho!}\Vol(\Nef(\pi_S(X))).
   \end{eqnarray*}
\end{proof}


\section{Del Pezzo surfaces}\label{delPezzo}

   We will now show that Theorem \ref{th chamber volume} enables us to compute the volumes of all
   Zariski chambers on del Pezzo surfaces, i.e., on
   surfaces $X$ with ample anticanonical divisor $-K_X$. The classification of
   del Pezzo surfaces is well known: either $X$ is the projective plane
   $\P^2$, or $\PtP$, or a blow-up $S_r$ of $\P^2$ in $1\le r\le 8$ points in general
   position\footnote{In this case \emph{in general position} means that no three of
   the points are collinear, no six lie on a conic and no eight on a cubic
   with one of them a double point.}. The degree of a del Pezzo surface
   is defined as the self-intersection of the anticanonical divisor. We have
   $$
      (-K_{\P^2})^2 = 9, \quad (-K_{\PtP})^2 = 8, \quad (-K_{S_r})^2 = 9-r.
   $$
\begin{lemma}\label{lemma contracting}
   Let $S_r$ be a del Pezzo surface with $1\le r\le8$ and $E$ a
   $(-1)$-curve on $S_r$. Then $E$ is contracted to a point on
   a del Pezzo surface $Y$ of degree $9-r+1$ by a birational
   morphism
   $$
      \pi_E : S_r\to Y.
   $$
   In particular
   \begin{eqnarray}
      N^1(S_r) &=& N^1(Y) \oplus \Z [E],\\
      -K_{S_r}&=&-K_Y-E.
   \end{eqnarray}
\end{lemma}

\begin{proof}
   For a curve $E$ of self-intersection $E^2=-1$ the adjunction formula
   combined with the ampleness of $-K_{S_r}$ reads
   $$
      0\le g(E) = 1+\frac12(E^2+EK_{S_r})\le0,
   $$
   implying that $E$ must be rational. By Castelnuovo's Contractibility
   Criterion, $E$ is contracted by a birational morphism $\pi_E$ to a point on
   a smooth surface $Y$. Regarding $S_r$ as the blow-up of $Y$ in a point
   with exceptional divisor $E$ renders the asserted identities obvious.
   It is now left to prove that $Y$ is del Pezzo of degree $9-r+1$. Consider the
   self-intersection
   $$
      (-K_Y)^2=(-K_{S_r}+E)^2=9-r+1 > 0.
   $$
   Furthermore, for any irreducible curve $C$ on $Y$ we have
   \begin{eqnarray*}
      (-K_Y\cdot C) &=& (\pi_E^\ast(-K_Y)\cdot\pi_E^\ast(C)) \\
      &=& ((-K_{S_r}-E)\cdot \pi_E^\ast(C)) \\
      &=& (-K_{S_r}\cdot \pi_E^\ast(C)) > 0.
   \end{eqnarray*}
   Consequently, $-K_Y$ is ample by the Nakai criterion.
\end{proof}

\begin{lemma}\label{ResultContraction}
   For $r\ge3$, contracting a $(-1)$-curve on $S_r$ results in the surface
   $S_{r-1}$. For a $(-1)$-curve $E$ on $S_2$, we have
   $\pi_E(S_2)=S_1$, if there exists a  $(-1)$-curve $E^\prime$ on $S_2$
   such that $(E\cdot E^\prime)=0$. Otherwise $\pi_E(S_2) = \P^1\times\P^1$.
\end{lemma}

\begin{proof}
   The assertion for $r\ge3$ follows immediately from Lemma
   \ref{lemma contracting} and the classification of del Pezzo surfaces, since
   any del Pezzo surface of degree $9-r+1$ is a surface $S_{r+1}$.

   Let now $r=2$ and consider $Y:=\pi_E(S_2)$. By the classification of
   del Pezzo surfaces,
   $Y$ is either $\PtP$ or $S_1$. Suppose there is a $(-1)$-curve $E'$
   on $S_2$ disjoint from $E$. Then $E'$ is the pull-back of a $(-1)$-curve
   on $Y$. Since $\PtP$ is minimal, $Y$ must be a surface $S_1$.

   If, on the other hand, there is no $(-1)$-curve on $S_2$ disjoint from
   $E$, then $Y$ cannot contain a $(-1)$-curve either:
   for such a curve $C$ on $Y$, the transform $\tilde C$ on a blow-up
   $X$ of $Y$ in a point $p$ is an
   irreducible curve with self-intersection $\tilde C^2=C^2-s^2$, where
   $s$ denotes the order of $C$ in the point $p$. Now, if $X$ is del Pezzo,
   then $\tilde C^2$ is at least $-1$, hence $s=0$, and $\tilde C$ is a
   $(-1)$-curve not intersecting the exceptional curve $E$.
   Therefore, contracting a curve $E$ having positive intersection
   with the other $(-1)$-curves on $S_2$ results in $\PtP$.
\end{proof}

   We now apply our knowledge about the behaviour of del Pezzo surfaces
   under contractions to calculate the chamber volumes.
\begin{proposition}\label{prop: Vol dP}
   Let $P$ be a big and nef divisor on a del Pezzo surface $S_r$, $1\le r\le 8$,  and let
   $\Null(P)=\set{E_1,\ldots,E_k}$.
   If $k\neq r-1$, the Zariski chamber $\Sigma_P$ corresponding to $P$ has the volume
   \begin{equation}
      \Vol({\Sigma_P})=\frac{(r-k+1)!}{(r+1)!}\Vol({\Nef(S_{r-k})}),
   \end{equation}
   where $S_0:=\P^2$.
   Otherwise, i.e., for $k=r-1$,
   \begin{equation}
      \Vol({\Sigma_P})=\begin{cases}
      \frac1{4(r+1)!}   & \text{, if $E_1,\ldots,E_k$
      form a maximal negative definite system}\\
      \frac{1}{6(r+1)!} & \text{, otherwise}.
   \end{cases}
\end{equation}
\end{proposition}
\begin{remark}\label{Nef Vol DP}\rm
   The following nef cone volumes $\Vol({\Nef(S_{r})})$ are calculated in \cite{D}.
   We will show in Example \ref{ex: dP} how to obtain these values as an application of Theorem 1.
   Note that Derenthal considers numbers $\alpha(S_r)$ which equal the nef
   cone volume multiplied by a dimensional factor $r+1$.
   In fact, $\alpha(S_r)$ is defined as the volume of the topmost \engq{slice}
   $\Nef(S_r)\cap K_{S_r}^{=1}$ of the polytope $\Nef(S_r)\cap K_{S_r}^{\le 1}$
   considered here.

   \begin{center}
      \begin{tabular}[t]{r|*{7}{c}c}\tline\centering
         $r$                                            &1 & 2& 3 & 4 & 5 & 6 & 7 & 8 \\ \tline
         $\Vol(\mathcal C_{\Nef(S_{r})})$&1/12&1/72&1/288&1/720&1/1080&1/840&1/240&1/9\\
         \tline
      \end{tabular}
   \end{center}
   Furthermore, $\Vol({\Nef(\P^2)})=\tfrac{1}{3}$ and
   $\Vol({\Nef(\P^1\times\P^1)})=\tfrac{1}{8}$.
   The proposition thus gives sufficient information to calculate
   the volumes of all Zariski chambers on del Pezzo surfaces
   (see below).
\end{remark}

\begin{proof}[Proof of Proposition \ref{prop: Vol dP}]

   For $k>r-1$ the asserted volume formula is a direct consequence
   of Theorem \ref{th chamber volume} together with Lemma
   \ref{ResultContraction} applied $k$ times.
   The contraction of $k=r$ pairwise disjoint $(-1)$-curves on $S_r$ results
   in a del Pezzo surface with Picard number 1, i.e., in $\P^2$. Thus in this
   case the result, again, follows immediately from Theorem \ref{th chamber volume}.
   In case $k=0$ the chamber $\Sigma_P$ is the nef chamber, whereby
   the assertion turns out to be trivial.

   Now, let us consider the remaining case $k=r-1$. Again, the formula
   essentially follows from Theorem
   \ref{th chamber volume}. What is still left to do is to establish
   whether \linebreak $E_1^\bot\cap\ldots\cap E_{r-1}^\bot\cap \Nef(S_r)$ is
   identified by $\pi^\ast$ with $\Nef(S_1)$ or with $\Nef(\PtP)$.
   The proof of Lemma \ref{ResultContraction} implies
   that the transform of every $(-1)$-curve on the surface $\pi_{E_1,\ldots,E_{r-1}}(S_r)$
   resulting from the contraction of $E_1,\ldots,E_{r-1}$ is itself a $(-1)$-curve on $S_r$
   not intersecting any of the $E_i$. It therefore forms a negative definite system together with
   the curves $E_i$. So, $\pi_{E_1,\ldots,E_{r-1}}(S_r)$ equals $\PtP$ if and only if
   $E_1\ldots,E_{r-1}$ form a maximal negative definite system and otherwise equals $S_1$.
\end{proof}

   In \cite{BFN} an algorithm was introduced that computes
   the number of Zariski chambers on a smooth surface with known
   negative curves by determining the number of negative definite
   principal submatrices of the intersection matrix of all negative curves.
   This algorithm can easily be modified in such a way that it returns the
   number of negative definite principal submatrices of a given size $s$.
   This number evidently equals the number of Zariski chambers $\Sigma_P$ whose
   support $\Null(P)$ contains $s$ curves.

   Note that for the chambers $\Sigma_P$ of support size $r-1$ the volume varies depending
   on whether the contraction of the curves in $\Null(P)$ yields the surface $\PtP$,
   or the surface $S_1$. From the algorithm we only obtain the overall number of chambers
   of a given support size. However, it is easy to show that on any of the
   surfaces $S_r$ exactly one third of the occurring chambers of support size $r-1$
   are of the first type: contracting any $r-2$ curves from Null$(P)$ results in a surface $S_2$,
   whose $(-1)$-curves have intersection matrix
   $$
      \matr{-1&1&1\\
      1 &-1&0\\
      1&0&-1}.
   $$
   Consequently the contraction of the first curve results in $\PtP$ and
   contracting either of the other two yields $S_1$.
   Since the surface resulting from iterated contraction of
   several $(-1)$-curves is independent of the order
   of contractions, the ratio between chambers of first to
   second type is one to two.

   The numbers and volumes of Zariski chambers on del Pezzo
   surfaces are displayed in tables \ref{S1} to \ref{S8} where
   the first and second columns indicate the support size $k$ and the
   surface $\pi_k(S_r)$ obtained by contracting the curves in $\Null(P)$.

   \begin{table}[ht]\centering
      \begin{tabular}[t]{cc|cc}\tline
         k &  $\pi_k(S_r)$   & number &   $\qquad\Vol\left(\Sigma_P\right)\qquad$\\\tline
         0 &  $S_1$          &  1  &          1/12   \\
         1 &  $\P^2$         &  1  &          1/6   \\\tline
      \end{tabular}
      \caption{Zariski chamber volumes on  $S_1$ \label{S1}}

   \end{table}

   \begin{table}[ht]\centering
      \begin{tabular}[t]{cc|cc}\tline
         k &  $\pi_k(S_r)$   & number &    $\qquad\Vol\left(\Sigma_P\right)\qquad$\\\tline
         0 &  $S_2$          &  1  &          1/72   \\
         1 &  $S_1$          &  2  &          1/36   \\
         1 &  $\PtP$         &  1  &          1/24   \\
         2 &  $\P^2$         &  1  &          1/18   \\\tline
      \end{tabular}
      \caption{Zariski chamber volumes on  $S_2$\label{S2}}

   \end{table}

   \begin{table}[ht]\centering
      \begin{tabular}[t]{cc|cc}\hline
         k &  $\pi_k(S_r)$     & number & $\qquad\Vol\left(\Sigma_P\right)\qquad$\\\tline
         0 &  $S_3$       &  1  &          1/288   \\
         1 &  $S_2$       &  6  &          1/288   \\
         2 &  $S_1$       &  6  &          1/144   \\
         2 &  $\PtP$      &  3  &          1/96  \\
         3 &  $\P^2$      &  2  &          1/72   \\\tline
      \end{tabular}
      \caption{Zariski chamber volumes on  $S_3$\label{S3}}

   \end{table}

   \begin{table}[ht]\centering
      \begin{tabular}[t]{cc|cc}\hline
         k &  $\pi_k(S_r)$   & number & $\qquad\Vol\left(\Sigma_P\right)\qquad$\\\tline
         0 &  $S_4$          &  1  &          1/720   \\
         1 &  $S_3$          &  10 &          1/1440  \\
         2 &  $S_2$          &  30 &          1/1440  \\
         3 &  $S_1$          &  20 &          1/720   \\
         3 &  $\PtP$         &  10 &          1/480   \\
         4 &  $\P^2$         &   5 &          1/360   \\\tline
      \end{tabular}
      \caption{Zariski chamber volumes on  $S_4$\label{S4}}

   \end{table}

   \begin{table}[ht]\centering
      \begin{tabular}[t]{cc|cc}\tline
         k &  $\pi_k(S_r)$   & number & $\qquad\Vol\left(\Sigma_P\right)\qquad$\\\tline
         0 &  $S_5$          &  1   &          1/1080 \\
         1 &  $S_4$          &  16  &          1/4320\\
         2 &  $S_3$          &  80  &          1/8640 \\
         3 &  $S_2$          &  160 &          1/8640 \\
         4 &  $S_1$          &  80  &          1/4320\\
         4 &  $\PtP$         &  40  &          1/2880 \\
         5 &  $\P^2$         &  16  &          1/2160 \\\tline
      \end{tabular}
      \caption{Zariski chamber volumes on  $S_5$ \label{S5}}

   \end{table}

   \begin{table}[ht]\centering
      \begin{tabular}[t]{cc|cc}\hline
         k &  $\pi_k(S_r)$   & number & $\qquad\Vol\left(\Sigma_P\right)\qquad$\\\tline
         0 &  $S_6$          &  1    &          1/840 \\
         1 &  $S_5$          &  27   &          1/7560 \\
         2 &  $S_4$          &  216  &          1/30240\\
         3 &  $S_3$          &  720  &          1/60480\\
         4 &  $S_2$          &  1080 &          1/60480\\
         5 &  $S_1$          &  432  &          1/30240\\
         5 &  $\PtP$         &  216  &          1/20160\\
         6 &  $\P^2$         &  72   &          1/15120\\\tline
      \end{tabular}
      \caption{Zariski chamber volumes on  $S_6$\label{S6}}

   \end{table}

   \begin{table}[ht]\centering
      \begin{tabular}[t]{cc|cc}\tline
         k &  $\pi_k(S_r)$   & number & $\qquad\Vol\left(\Sigma_P\right)\qquad$\\\tline
         0 &  $S_7$          &  1  &          1/240\\
         1 &  $S_6$          &  56 &          1/6720\\
         2 &  $S_5$          &  765 &          1/60480\\
         3 &  $S_4$          &  4032&          1/241920\\
         4 &  $S_3$          &  10080&          1/483840\\
         5 &  $S_2$          &  12096&          1/483840\\
         6 &  $S_1$          &  4032&          1/241920\\
         6 &  $\PtP$         &  2016&          1/161280\\
         7 &  $\P^2$         &  576&          1/120960\\\tline
      \end{tabular}
      \caption{Zariski chamber volumes on  $S_7$\label{S7}}

   \end{table}

   \begin{table}[ht]\centering
      \begin{tabular}[t]{cc|cc}\tline
         k &  $\pi_k(S_r)$   & number & $\qquad\Vol\left(\Sigma_P\right)\qquad$\\\tline
         0 &  $S_8$          &  1  &          1/9   \\
         1 &  $S_7$          &  240  &          1/2160   \\
         2 &  $S_6$          &  6720  &          1/60480   \\
         3 &  $S_5$          &  60480  &          1/544320   \\
         4 &  $S_4$          &  241920  &          1/2177280   \\
         5 &  $S_3$          &  483840  &          1/4354560   \\
         6 &  $S_2$          &  483840  &          1/4354560   \\
         7 &  $S_1$          &  138240  &          1/2177280   \\
         7 &  $\PtP$         &  69120  &          1/1451520   \\
         8 &  $\P^2$         &  17280  &          1/1088640   \\\tline
      \end{tabular}
      \caption{Zariski chamber volumes on  $S_8$\label{S8}}

   \end{table}


\section{Big anticanonical surfaces}\label{sect: Bl up}


\subsection{Finiteness of nef chamber volume}

   As we have seen, the calculation of Zariski chamber volumes $\Vol(\Sigma_P)$ reduces
   to calculations of nef cone volumes on surfaces resulting from contraction
   of curves in $\Null(P)$. For that reason we for now turn our attention to nef chamber
   volumes.
   Our first question is: which surfaces have finite nef cone volume?

   First note that
   $\kappa(X)=-\infty$ is a necessary condition for the nef cone on a surface $X$ to
   have finite volume. Otherwise the anticanonical divisor on the
   (in this case unique) minimal model $X'$ for $X$
   would be nef with non-negative self-intersection by virtue of the well known
   classification of smooth algebraic surfaces. But then $-K_X\cdot K_X\le 0$, hence $X'$
   (and thus $X$ itself) would have infinite nef cone volume.
   Our aim is now to show:

\begin{proposition}\label{criterion finiteness}
   A smooth projective surface $X$ has finite nef cone volume if and only if its
   anticanonical divisor $-K_X$ is big.
\end{proposition}

   For the proof we first need a statement on convex cones,
   which may be seen as an
   \engqq{in vitro} version of Kleiman's ampleness criterion:

\begin{lemma}\label{lemma:convex-cones}
   Let $C\subset\R^n$ be a closed cone, and let
   $$
      C^*=\set{x\in\R^n\with x\cdot c\ge 0\mbox{ for all }
         c\in C}
   $$
   be its dual cone (with respect to a fixed non-degenerate bilinear
   form).
   We have the following characterization of its interior:
   $$
      \interior(C^*)=\set{x\in\R^n\with x\cdot c>0\mbox{ for all }
         c\in C\setminus\set 0}
   $$
\end{lemma}

\begin{proof}
   Denote by $D$ the set on the right-hand side. We show first
   that $\interior(C^*)\subset D$.
   Suppose to the contrary there exists a point $x_0\in \interior(C^*)$
   not in the set $D$, i.e., $x_0\cdot c=0$ for some non-zero $c\in C$.
   Consider the non-zero linear function
   \be
      \phi_c: \R^n & \to & \R \\
      x & \mapsto & x\cdot c \,.
   \ee
   It has a zero in $x_0$, hence must take negative values on points
   in any neighbourhood $U$ of $x_0$. However, if we choose $U$ sufficiently small, then it is
   contained in $\interior(C^*)\subset C^*$, and hence $x\cdot
   c\ge 0$ for $x\in U$. This is a contradiction.

   We now show that $D$ is an open set. As we already know that
   $\interior(C^*)\subset D\subset C^*$, this will conclude the proof.
   Let then $S\subset\R^n$ be the 1-sphere (with respect to any
   fixed norm). We are done if either $D$ or $C\cap S$ are empty.
   Otherwise
   consider for $d\in D$ the linear function
   \be
      \psi_d: C\cap S & \to & \R \\
      x & \mapsto & x\cdot d \,.
   \ee
   It has only positive values and assumes a minimum on the
   compact set $C\cap S$, hence there is a $\delta>0$ such that
   $x\cdot d\ge\delta$ for all $x\in C\cap S$.
   Consequently there is a neighbourhood of $d$ in $\R^n$, all of
   whose elements have positive product with every
   $x\in C\cap S$, and hence with every $x\in C$.
\end{proof}

   If $C$ is the Mori-cone $\NE(X)$ of a
   smooth complex variety $X$, then its dual is by definition the
   nef cone $\Nef(X)$. Using now that by Kleiman's theorem
   \cite[Theorem 1.4.23]{L}
   the ample cone is the interior of the nef cone, Lemma~\ref{lemma:convex-cones}
   recovers Kleiman's ampleness criterion
   \cite[Theorem 1.4.29]{L}:
   $$
      \Amp(X)=\set{D\in N^1_\R(X)\with D\cdot \xi>0\mbox{ for all non-zero } \xi\in\NE(X)}
   $$
   For our present purposes we will need the dual statement:
   If $C$ is the
   nef cone of a smooth projective surface, then its dual
   is by definition the Mori cone, whose
   interior is by \cite[Theorem 2.2.26]{L} the big cone, and
   Lemma~\ref{lemma:convex-cones} yields:

\begin{corollary}\label{corollary:big-cone}
   For any smooth projective surface $X$ we have
   $$
      \Bigcone(X)=\set{D\in N^1_\R(X) \with D\cdot D' > 0 \mbox{ for all non-zero } D' \in \Nef(X)}.
      \label{eq:bigcone}
   $$
\end{corollary}

\begin{proof}[Proof of Proposition \ref{criterion finiteness}]
   The nef cone on $X$ has finite volume if and only if the hypersurface
   $(-K_X)^{=1}$ intersects each of its rays. This is the case if and only if
   for each nef divisor $D$ there exists a positive rational number $d$ such that
   $dD\cdot (-K_X)=1$, i.e., every nef divisor must have positive intersection with
   $-K_X$. The assertion follows now from
   Corollary~\ref{corollary:big-cone}.
\end{proof}

\begin{lemma}
   Let $X$ be a smooth projective surface and let $-K_X$ be big.
   Then $\NE(X)$ is finitely generated.
\end{lemma}

\begin{proof}
   This is shown in \cite[Lemma 6]{CS} for rational surfaces.
   Note however that   the given proof of the finite generation
   does not depend on rationality:
   for any $\varepsilon > 0$ and any ample divisor $H$ by the cone
   theorem we find finitely many irreducible curves $C_i$ such that
   $$
      \NE(X) = \NE(X)_{(-K_X-\varepsilon H)^{\le0}} + \sum \R^+_0\cdot [C_i].
   $$
   Furthermore, for a sufficiently small $\varepsilon>0$ the stable base locus
   $\mathbb B(-K_X-2\varepsilon H)$ equals the augmented base locus
   $\mathbb B_+(-K_X)$, which is just $\Null(P)$, where $P$ denotes
   the positive part in the Zariski decomposition of $-K_X$,
  see~\cite[Example 1.11]{ELMNP}.
  (The notion of augmented base locus was introduced in
  \cite{ELMNP} and is motivated by \cite{N}; we recommend
  \cite{ELMNP} or \cite[Sect.~10.3]{L} for an exposition.)

   Now, any irreducible curve $C$ in $(-K_X-\varepsilon H)^{\le0}$
   has negative intersection with $-K_X-2\varepsilon H$, thus is
   an element of $\mathbb B(-K_X-2\varepsilon H)= \Null(P)$. Since
   the intersection matrix of the curves in $\Null(P)$ is negative
   definite, there can be at most $\rho-1$ such curves.
\end{proof}

\begin{theorem}\label{th 2}
   Let $X$ be a smooth surface with big anticanonical divisor
   $-K_X$ and Picard number $\rho$.
   There exists a divisor
   $D\in N^1_\R(X)$ such that the nef cone volume is given by
   $$
      \Vol(\Nef(X)) = \frac1{\rho}\cdot \sum_{E}{(D\cdot E)\cdot\Vol(\Nef(\pi_E(X)))},
   $$
   where the sum is taken over all $(-1)$-curves $E$ in $X$.
\end{theorem}
   Note that together with Proposition \ref{criterion finiteness} this yields Theorem 1 from the introduction.

\begin{proof}
   The proof consists of two parts. First we argue that there is a divisor class
   $D\in N^1_\R(X)$ with $-K_X\cdot D=1$ such that $D$ has intersection zero with
   all irreducible curves whose self-intersection is strictly less than $-1$. In the
   second part we show that for such an element $D$ the claimed identity holds.

   By assumption $-K_X$ is big, thus there exists a representation
   $$
      -K_X = A + B
   $$
   with $A$ an ample $\Q$-divisor and $B$ an effective $\Q$-divisor. Now, let $C$ be an irreducible curve on $X$ with
   $C^2\le -2$. By adjunction we have
   $$
      0\ge -K_X\cdot C = AC + BC.
   $$
   The ampleness of $A$ implies that $A\cdot C$ is strictly positive, showing that $B\cdot C$
   must be strictly negative. Now, being effective, $B$ admits a Zariski decomposition
   $$
      B = P_B + N_B
   $$
   in a nef part $P_B$ and a divisor $N_B$, whose components have negative definite
   intersection matrix. Any curve with $C\cdot B< 0$ thus must be one of the components of $N_B$,
   in other words, $C$ must be an element of $\Neg(B)$. Since the intersection matrix of the
   curves in $\Neg(B)$ is negative definite, $\Neg(B)$ can contain at most $\rho -1$ curves.
   By the same token, there exists a big and nef divisor $P$ on $X$
   with $\Null(P)=\Neg(B)$ (see \cite[Proposition 1.1]{BFN}). In particular we have $P \cdot C = 0$
   for all curves $C$ with $C^2\le -2$. Note that the nefness of $P$ together with
   the bigness of $-K_X$ implies the inequality $-K_X\cdot P > 0$. We can therefore set
   $$
      D:= \frac1{-K_X\cdot P}\cdot P,
   $$
   obtaining a divisor with the desired properties.

   We prove the volume formula by decomposing the polytope
   \linebreak $\mathcal P_X := \Nef(X)\cap(-K_X)^{\le1}$ into pyramids
   with vertex $D$ and the facets of $\mathcal P_X $ as bases.
   Since $D$ is nef by construction and contained in the
   hypersurface $(-K_X)^{=1}$, the polytope's volume is just the sum
   of the volumes of all the pyramids in the decomposition.
   Note that $D$ in addition lies inside all the hypersurfaces
   $C^\bot$ for curves with self-intersection less than $-1$.
   The corresponding pyramids thus have volume 0, hence the nef
   cone volume is just   the sum of volumes of the pyramids with
   bases $\mathcal P_X\cap E^\bot$ for $(-1)$-curves $E$.
   As we have seen, these bases correspond to the
   $(\rho-1)$-dimensional polytopes $\mathcal P_{\pi_E(X)}$,
   thus have the same volume as the nef cone on the surface
   $\pi_E(X)$ resulting from the contraction of $E$.
   The asserted formula follows once we have shown that
   the factor $D\cdot E$ represents the height of the pyramid
   corresponding to $E^\bot$. This is indeed the case: for a
   vector space basis $E_1,\ldots,E_{\rho-1}$ of  $N^1(\pi_E(X))$,
   the vectors $\pi^\ast(E_1),\ldots, \pi^\ast(E_{\rho-1}),-E$ form
   a basis of $N^1(X)$. In this basis the vector $(0,\ldots,0,1)$
   is a normal vector to the hypersurface $E^\bot$. Let $D$
   have a representation $(\alpha_1,\ldots,\alpha_{\rho-1},\alpha)$
   in this basis. Then, since   $E\cdot \pi^\ast(E_i)=0$,  the number $\alpha$
   on the one hand is just the intersection product $D\cdot E$
   and on the other hand its absolute value $|\alpha|$ is the distance of the point $D$ to
   the hypersurface $E^\bot$, i.e., the height of the pyramid in
   question. By our construction, the divisor $D$ is nef, therefore $|\alpha|=\alpha$.
\end{proof}

\begin{example}\label{ex: dP}\rm
   Let $S_r$ be a del Pezzo surface with $3\le r\le8$. The
   decomposition into ample and effective part in the proof
   is just the trivial decomposition
   $$
      -K_X = A + B = -K_X + 0,
   $$
   hence $\Neg(B)$ is empty. Therefore,
   $$
      \Null(-K_{S_r})=\Neg(B)=\emptyset.
   $$
   Following the proof above, we set $D:= \tfrac1{9-r}(-K_{S_r})$ and
   obtain
   \begin{eqnarray*}
      \Vol(\Nef(S_r)) &=& \frac1\rho\sum_E DE\cdot \Vol(\Nef(\pi_E(S_r)))\\
      &=& \frac1{r+1}\sum_E \frac1{9-r}\cdot \Vol(\Nef(S_{r-1}))\\
      &=& \frac{N_r}{(r+1)(9-r)}\Vol(\Nef(S_{r-1})),
   \end{eqnarray*}
   where $N_r$ denotes the number of $(-1)$-curves on the del Pezzo surface $X_r$.
   This formula for the nef cone volume turns out to be the same as calculated in \cite{D}.
\end{example}

   The existence of a formula for the nef cone volume does not necessarily imply
   that it is easy to calculate for any given big anticanonical surface $X$. Knowledge
   of the negative curves on $X$ and on the surfaces resulting from contracting
   $(-1)$-curves is key to the calculation: with this knowledge, an inductive
   calculation is possible, since successive contraction of $(-1)$-curves eventually
   yields a minimal surface with $\kappa=-\infty$, and the nef cone volumes on
   these surfaces are easy to compute.

   Testa, V\'arilly-Alvarado, Velasco in \cite{TVV} list surfaces known to
   have big anticanonical divisor, e.g.  rational surfaces with $K_X^2>0$ or
   blow-ups of Hirzebruch surfaces $X_e$, $e\ge1$, in points that lie on the union
   of the section $C$ with $e+1$ fibers. Furthermore, they give a classification
   for surfaces which are obtained as blow-ups of $\P^2$ in $r$ points
   and have big and effective anticanonical divisor (see \cite[Theorem 3.4]{TVV}). For such a surface
   $X$ one has either
   \begin{itemize}
   \item
      $K_X^2>0$, i.e., $r\le8$, or
   \item
      $a$ of the blown-up points lie on a line, the other $b=r-a$ points
      lie on a irreducible conic, and either $ab=0$ or $\frac1a +\frac4b>1$, or
   \item
      the blown-up points lie on the union of three lines $L_1,L_2,L_3$ with
      $a_i$ of them exclusively on $L_i$, and either $a_1a_2a_3=0$ or
      $\frac1{a_1}+\frac1{a_2}+\frac1{a_3}>1$.
   \end{itemize}
   Note that these surfaces are no longer del Pezzo as soon as curves with
   self-intersection less than $-1$ occur. In the first case this happens
   if the blown-up points are not in general position, in the second case if
   either $a\ge3$ or $b\ge6$, and in the third case if one of the lines $L_i$
   contains at least three of the blown-up points, or if three of the blown-up
   points on different lines $L_i$ are collinear.
   Such surfaces thus provide interesting examples for the application of
   theorem \ref{th 2}.
   In order to illustrate our method we consider
   non-del Pezzo surfaces from the second class:
   we determine the nef cone volume and the volumes of the Zariski chambers
   for blow-ups in $r$ points on a line -- among all blow-ups of $\P^2$
   these are in a sense the other extreme to del Pezzo surfaces
   (see Sect.~\ref{s:big}).
   Finally, we do the analogous computations for blow-ups at infinitely near points
   (see Sect.~\ref{subsect:infinitely-near}).
   We plan to study further
   surfaces with big anticanonical class
   in a subsequent paper.


\subsection{Blow-ups of points on a line in $\P^2$}\label{s:big}

   Let $L$ be a line in $\P^2$ and $p_1,\ldots, p_r$ points on $L$. We consider
   the blow-up
   $$
      \pi: X_L^r\to\P^2
   $$
   in these points and denote the strict transform of $L$ by $\tilde{L}$. Furthermore,
   let $L'$ denote the transform of a general line in $\P^2$.
\begin{proposition}\label{prop: contr XL}
   The negative curves on $X_L^r$ are $\tilde{L}$ and the exceptional curves
   $E_1,\ldots,E_r$. Contracting any of the curves $E_i$ results in
   a surface $X_L^{r-1}$.
\end{proposition}

\begin{proof}
   Suppose there exists a curve $C= dL' -\sum_{i=1}^r m_i E_i$ with negative
   self-intersection neither equal to $\tilde L$ nor to one of the exceptional curves.
   By adjunction we have $-K_{X_L^r}\cdot C\le 1$, or
   $$
      3d-\sum{m_i}\le 1.
   $$
   Since $\tilde L$ corresponds to the class $L'-\sum{m_iE_i}$ and has
   non-negative intersection $d-\sum{m_i}$ with $C$ we have
   $$
      1\ge 3d-\sum{m_i} = 2d + (d-\sum{m_i}) \ge 2d \ge 2,
   $$
   a contradiction.
   The second assertion is obvious.
\end{proof}

   We now determine the nef cone volume of $X_L^r$ and
   the volumes of all Zariski chambers on this surface.

\begin{proposition}
   For any $r\ge1$, the nef cone volume on $X_L^r$ is given by
   \begin{eqnarray*}
      \Vol(\Nef(X_L^r)) &=& \frac1{2r+2} \Vol(\Nef(X_L^{r-1})) \\
      &=& \left(\frac12\right)^r\cdot\frac1{(r+1)!}\cdot \frac13.
   \end{eqnarray*}

\end{proposition}
\begin{proof}
   Following the proof of Theorem \ref{th 2} we need to determine a divisor $D$ on
   $X_L^r$ with $\tilde L\cdot D=0$ and $-K_{X_l^r}\cdot D=1$. These are the only
   conditions since $\tilde L$ is the only curve that can have self-intersection
   less than $-1$. The divisor class
   $$
      D = \frac12\cdot (L'-E_1)
   $$
   satisfies these conditions, and moreover lies in the hyperplanes $E_i^\bot$
   for all $i=2,\ldots,r$.
   Therefore, by Proposition \ref{prop: contr XL}
   \begin{eqnarray*}
      \Vol(\Nef(X_L^r)) &=& \frac 1\rho\sum_{i=1}^r (E_i\cdot D)\Vol(\Nef(X_L^{r-1}))\\
      &=& \frac 1{r+1}\cdot \frac12\Vol(\Nef(X_L^{r-1})).
   \end{eqnarray*}
   Now, the second identity follows inductively using the fact that
   $\Vol(\Nef(\P^2))=\tfrac13$.
\end{proof}

   The following statements about the remaining Zariski chambers are immediate
   consequences of Theorem \ref{th chamber volume}.

\begin{proposition}
   If $r\ge 3$, then for a big an nef divisor $P$ on $X_L^r$ the set $\Null(P)$
   either contains $\tilde L$ and $\Vol(\Sigma_P)=\infty$, or $\Null(P)$
   consists of $s$ exceptional curves and
   \begin{eqnarray*}
      \Vol(\Sigma_P) &=& \frac {(r+1-s)!}{(r+1)!} \Vol(\Nef(X_L^{r-s})) \\
      &=& \left(\frac12\right)^{r-s}\cdot\frac1{(r+1)!}\cdot \frac13.
   \end{eqnarray*}
   If $1\le r\le2$, then $X_L^r$ is the del Pezzo surface $S_r$ with chamber volumes
   according to Proposition \ref{prop: Vol dP}.
\end{proposition}


\subsection{Blow-ups of infinitely near points}\label{subsect:infinitely-near}

   As a final illustration of the applicability of our technique we
   consider surfaces obtained by iteratively blowing up points infinitely near to $\P^2$:
   we start by picking a line $L$ in $\P^2$ and a point $p_1$ on $L$.
   Blowing up $p_1$ yields the del Pezzo surface $S_1$ on which the strict
   transform $L_1$ of $L$ is an irreducible curve of self-intersection zero.
   On the exceptional curve $E_1$ pick the point $p_2$ corresponding to the tangential
   direction of $L$ in $p_1$. We denote the blow-up of $S_1$ in $p_2$ by $X^\infty_2$.
   Now, on the exceptional curve $E_2$ of the second blow-up pick the point $p_3$ corresponding to the
   tangential direction of $L_1$ in $p_2$, blow it up, and denote the resulting
   surface by $X^\infty_3$. Repeating this process yields surfaces $X^\infty_r$ for
   all natural numbers $r\ge 2$. Note that on these surfaces the anticanonical class
   decomposes as $-K_{X_r^\infty}= (2L) + (L-E_1-\ldots-E_r)$ into a big and an effective divisor, hence is big.

\begin{proposition}
   For $r\ge 2$ the classes of negative curves on $X_r^\infty$ are
   \begin{itemize}
   \item $E_k-E_{k+1}$ for $1\le k\le r-1$,
   \item $E_r$, and
   \item $L-E_1-\ldots-E_r$.
   \end{itemize}
\end{proposition}
\begin{proof}
   By the construction of $X_r^\infty$, the class $L-E_1-\ldots -E_r$ contains
   an irreducible curve $L'$. Its self-intersection is $1-r$. Again by construction,
   the classes $E_k-E_{k+1}$ contain irreducible curves of self-intersection $-2$.
   Suppose there exists a negative curve $E$ on $X_r^\infty$ not listed above. Then $E$
   has a representation $E=dL-\sum_{i=1}^r m_iE_i$ and by adjunction the intersection
   with the anticanonical divisor is at most 1. By the irreducibility of $L'$, we obtain
   $$
      1\ge -K_{X_r^\infty} E = 3d - \sum m_i = 2d + L'E \ge 2,
   $$
   a contradiction.
\end{proof}

\begin{proposition}
   For any $r\ge2$, the nef cone volume on $X_r^\infty$ is given by
   \begin{eqnarray*}
      \Vol(\Nef(X_L^r)) &=& \frac1{2r(r+1)} \Vol(\Nef(X_{r-1}^\infty)) \\
      &=& \frac1{2^r\cdot r!(r+1)!}\cdot \frac13.
   \end{eqnarray*}
   There is exactly one additional chamber having finite volume, namely the chamber
   $\Sigma_P$ with $\Null(P)=\set{E_r}$. Its
   volume is
   \begin{eqnarray*}
      \Vol(\Sigma_P) &=& \frac{((r+1)- 1)!}{(r+1)!}\Vol(\Nef(\pi_{E_r}(X_r^\infty)))\\
      &=& \frac1 {r+1}\Vol(\Nef(X_{r-1}^\infty)).
   \end{eqnarray*}
\end{proposition}
\begin{proof}
   The equations defining the required divisor $D=dL-\sum a_iE_i$ in Theorem \ref{th 2}
   in this setting are
   \begin{eqnarray*}
      3d-a_1-\ldots -a_r &=& 1, \\
      d-a_1-\ldots -a_r &=& 0, \\
      a_j-a_{j+1} &=& 0 \quad \text{ for } 1\le j\le r-1.
   \end{eqnarray*}
   Consequently we set $D:=\frac12( L - \frac13\sum_{i=1}^r E_i)$. Then by the theorem,
   the nef chamber volume is given by
   \begin{eqnarray*}
      \Vol(\Nef(X_r^\infty)) &=& \frac 1{r+1}\cdot (D\cdot E_r)\cdot\Vol(\Nef(\pi_{E_r}(X_{r}^\infty)))\\
      &=& \frac1 {2r(r+1)}\cdot \Vol(\Nef(X_{r-1}^\infty)),
   \end{eqnarray*}
   and the second asserted identity follows inductively.

   The statement about the remaining Zariski chambers
   follows using
   Theorem~\ref{th chamber volume}.
\end{proof}



   \bigskip
   \small
   Tho\-mas Bau\-er,
   Fach\-be\-reich Ma\-the\-ma\-tik und In\-for\-ma\-tik,
   Philipps-Uni\-ver\-si\-t\"at Mar\-burg,
   Hans-Meer\-wein-Stra{\ss}e,
   D-35032~Mar\-burg, Germany.

   \nopagebreak
   \textit{E-mail address:} \texttt{tbauer@mathematik.uni-marburg.de}

   \bigskip
   David Schmitz,
   Fach\-be\-reich Ma\-the\-ma\-tik und In\-for\-ma\-tik,
   Philipps-Uni\-ver\-si\-t\"at Mar\-burg,
   Hans-Meer\-wein-Stra{\ss}e,
   D-35032~Mar\-burg, Germany.

   \nopagebreak
   \textit{E-mail address:} \texttt{schmitzd@mathematik.uni-marburg.de}


\end{document}